\documentclass[11pt, onecolumn, twosided, a4paper, reqno]{amsart}
\usepackage{latexsym,amsmath,amssymb,enumerate,hyperref,multicol}
\usepackage{amsfonts}
\usepackage{enumitem}
\usepackage{graphics,epstopdf}
\usepackage[pdftex]{graphicx}
\usepackage{geometry}
\numberwithin{equation}{section}
\newtheorem{thm}{Theorem}[section]

\begin{document}
\setcounter{page}{1}
\title[New approach to generalized Mittag-Leffler function]
{New approach to generalized Mittag-Leffler function via quantum calculus}
\author[R. Nadeem, M. Saif, T. Usman and A. H. Khan]{Raghib Nadeem, Mohd. Saif, Talha Usman$^{*}$ and Abdul Hakim Khan}
\address[Raghib Nadeem, Mohd. Saif and Abdul Hakim Khan]{Department of Applied Mathematics,
Faculty of Engineering and Technology, Aligarh Muslim University,
Aligarh--202002, India.\newline}

\address[ Talha Usman$^{*}$]{Department of Mathematics, School of Basic and Applied
Sciences, Lingaya's Vidyapeeth, Faridabad-121001, Haryana,
India.\newline}

\email{{raghibmaths2015@gmail.com; usmanisaif153@gmail.com;
\newline \qquad\qquad $^{*}$talhausman.maths@gmail.com; ahkhan.amu@gmail.com}}



\begin{abstract}
We aim to introduce a new extension of Mittag-Leffler function via $q$-analogue and obtained their significant properties including integral representation, $q$-differentiation, $q$-Laplace transform, image formula under $q$-derivative operators. We also consider some particular cases to give the applications of our main results.

\end{abstract}
\maketitle
\section{Overture}
The Swedish mathematician G\"{o}sta  Mittag-Leffler discovered a special function in 1903 (see, \cite{Gml,Gml1}) defined as

\begin{equation}\label{eq4.1}
E_{\eta}(u) = \sum\limits_{m=0}^{\infty}\frac{u^{m}}{\Gamma(\eta m + 1) \ m!},~~~  (\eta, u \in \mathbb{C} ; \Re(\eta)>0),
\end{equation}
where $\Gamma(.)$ is a classical gamma function \cite{Red}. The special function defined in (\ref{eq4.1}) is called Mittag Leffler function (MLf).

For the very first time, in 1905, A. Wiman \cite{Awim} firstly proposed the generalization of the MLf $E_{\eta}(u)$ as
\begin{equation}\label{eq4.2} E_{\eta,\kappa}(u) =
 \sum\limits_{m = 0}^{\infty} \frac{u^{m}}{\Gamma(\eta m + \kappa) \ m!},\quad (\eta, \kappa \in \mathbb{C}; ~\Re(\eta)> 0, \Re(\kappa)> 0).
\end{equation}

Subsequently, the generalized form of series (\ref{eq4.1}) and (\ref{eq4.2}) studied by Prabhakar \cite{Trp} in 1971  as:

\begin{equation}\label{eq4.3}
E_{\eta,\kappa}^{\sigma}(u) = \sum\limits_{m=0}^{\infty}\frac{u^{m} (\sigma)_{m}}{\Gamma(\eta m + \kappa) \ m!},~~~  (\eta, \kappa, \sigma\in \mathbb{C};~\Re(\eta)>0, \Re(\kappa)>0, \Re(\sigma)>0),
\end{equation}
where $(\sigma)_{m} = \frac{\Gamma(\sigma + m )}{\Gamma(\sigma)}$ denotes the Pochhammer symbol \cite{Red}.

The Mittag-Leffler function plays a vital role in the solution of fractional order differential equations and  fractional order integral equations. It has recently become a subject of rich interest in the field of fractional calculus and its applications and  nowadays some mathematicians consider to refer the classical Mittag-Leffler function as the \textit{Queen Function} in the Fractional Calculus. An enormous amount of research in the theory of Mittag-Leffler functions has been published in the literature. For detailed account of the various generalizations, properties and applications of the MLf  readers may refer to the literatures \cite{Pms,Kss,Pi,Sks,Sal,Pi1,Mgr,Kaa}.

The $q$-calculus is the $q$-extension of the ordinary calculus. The theory of $q$-calculus operators in recent past have been applied in the areas of ordinary fractional calculus, optimal control problem, in finding solutions of the $q$-difference and $q$-integral equations and $q$-transform analysis.

In 2009, Mansoor \cite{Man} has proposed a new form of $q$-analogue of the Mittag-Leffler function is given as:

\begin{equation}\label{eq4.6} e_{\eta,\kappa}(u;q) =
\sum\limits_{m = 0}^{\infty} \frac{u^{m}}{\Gamma_{q}(\eta m + \kappa)} \ (|z|<(1-q)^{-\alpha}),
\end{equation}
where $ \eta > 0 ,\kappa \in \mathbb{C}$.

Recently, Sharma and Jain \cite{Skr} introduced the $q$-analogue of generalized MLf as given underneath:

\begin{equation}\label{eq4.7}
E_{\eta,\kappa}^{\sigma}(u;q) = \sum\limits_{m=0}^{\infty}\frac{(q^{\sigma};q)_{m}}{(q;q)_{m}}
\frac{u^{m}}{\Gamma_{q}(\eta m + \kappa)},
\end{equation}
\begin{equation*}
(\eta, \kappa, \sigma\in \mathbb{C};~\Re(\eta)>0, \Re(\kappa)>0, \Re(\sigma)>0 , |q|<1 ).
\end{equation*}

\section{Prelude }
In the theory of $q$-series (see\cite{Gasp}), for complex $\lambda$ and $0<q<1$, the $q$-shifted factorial is defined as follows:
\begin{equation} \label{eq4.8}
(\lambda;q)_{m} =
\begin{cases}
 1 & ; m=0,\\
 (1 - \lambda) (1 - \lambda q) \ldots (1 - \lambda q^{m - 1}) & ; m\in \mathbb{N}
\end{cases}
\end{equation}
which is equivalent to
\begin{equation}\label{eq4.9}
(\lambda;q)_{m} = \frac{(\lambda;q)_{\infty}}{(\lambda q^{m};q)_{\infty}}
\end{equation}
and its extension naturally as:
\begin{equation}\label{4.10}
(\lambda;q)_{\eta} = \frac{(\lambda;q)_{\infty}}{(\lambda q^{\eta};q)_{\infty}},~~ \eta\in \mathbb{C},
\end{equation}
where the principal value of $q^{\eta}$ is taken.
\vspace{.1cm}

For $s,t\in\mathbb{R}$ the $q$-analogue of the exponent $(s-t)^{m}$ is
\begin{equation}\label{eq4.12}
(s-t)^{(m)} =
\begin{cases}
1 & ; m = 0 \\
 \prod\limits_{i = 0}^{m-1} (s-tq^{i}) &  ;m\neq 0
\end{cases}
\end{equation}
and connected by the following relationship
$$(s-t)^{(m)} = s^{m}(t/s;q)_{m}, ~~~~(s\neq0).$$

Obviously, its expansion for $\tau \in \mathbb{R}$ as
\begin{equation}\label{eq4.13}
(s-t)^{(m)} = s^{m}\frac{(t/s;q)_{\infty}}{(q^{\tau}t/s;q)_{\infty}},~~~~\
(s;q)_{\tau}=\frac{(s;q)_{\infty}}{(s q^{\tau};q)_{\infty}}.
\end{equation}

Note that
$$(s-t)^{(\tau)} = s^{\tau}(t/s;q)_{\tau}.$$

The $q$-analogue of binomial coefficient is defined for $s,t >0$ as
\begin{equation}\label{eq4.14}
{\binom{s}{t}}_{q} = \frac{[s]_{q}}{[t]_{q}[s-t]_{q}}
= \frac{(q;q)_{s}}{(q;q)_{t}(q;q)_{s - t}}
= {\binom{s}{s-t}}_{q}.
\end{equation}

The definition can be generalized in the following way. For arbitrary complex $\tau$ we have
\begin{equation}\label{eq4.15}
{\binom{\tau}{m}}_{q} = \frac{(q^{-\tau};q)_{m}}{(q;q)_{m}}
(-1)^{m} q^{\tau m - \binom{m}{2}}
= \frac{\Gamma_{q}(\tau + 1)}{\Gamma_{q}(m + 1)\Gamma_{q}(\tau - m + 1)}
\end{equation}
where $\Gamma_{q}(u)$ is the q-gamma function.
\vspace{.1cm}

The $q$-gamma and $q$-beta functions(\cite{Gasp}) are defined by
\begin{equation}\label{eq4.16}
\Gamma_{q}(u) = \frac{(q;q)_{\infty}}{(q^{u};q)_{\infty}}(1 - q)^{1-u} = (1 - q)^{(u - 1)} (1 - q)^{1-u},
\end{equation}
for $ u\in \mathbb{R} \setminus \{0,-1,-2,-3,\ldots\}; |q|<1.$
\vspace{.1cm}

 Clearly,
$$\Gamma_{q}(u + 1) = [u]_{q} \Gamma_{q}(u) $$
and
\begin{equation}\label{eq4.17}
B_{q}(\eta, \kappa) =\frac{\Gamma_{q}(\eta)\Gamma_{q}(\kappa)}{\Gamma_{q}(\eta + \kappa)} = \int_{0}^{1} u^{\eta -1} \frac{(qu;q)_{\infty}}{(q^{\kappa}u;q)_{\infty}} d_{q}u
= \int_{0}^{1}u^{\eta - 1} (uq;q)_{\kappa - 1} d_{q}u,
\end{equation}
 $$(\Re(\eta),\Re(\kappa) >0).$$
Further, the $\Gamma_{q}(u)$ satisfies the functional equation
\begin{equation}\label{eq4.18}
\Gamma_{q}(u+1) = \frac{1-q^{u}}{1-q}\Gamma_{q}(u)
\end{equation}

Also, the $q$-difference operator and $q$-integration of a function $f(u)$ defined on a subset of $\mathbb{C}$ are given by \cite{Gasp} respectively.
\begin{equation}\label{eq4.19}
D_{q}f(u) = \frac{f(u) - f(uq)}{u(1 - q)} \ ~(u\neq 0, q\neq 1), (D_{q}f)(0) = \lim_{u\rightarrow 0}(D_{q}f)(u)
\end{equation}
and
\begin{equation}\label{eq4.19I}
\int_{0}^{u} f(t) d(t;q) = u (1 - q) \sum\limits_{m = 0}^{\infty} q^{m} f(u q^{k}).
\end{equation}
\section{Extended $q$-Mittag-Leffler function and their properties}
In this section, we extend the definition (\ref{eq4.7}) by introducing the following relation for
$(q^{c}, q)_{m}$
\begin{equation}\label{eq4.20}
\frac{(q^{c};q)_{m}}{(q^{\sigma};q)_{m}} = \frac{B_{q}(\sigma + m, c - \sigma)}{B_{q}(\sigma, c - \sigma)}.
\end{equation}
Now, we define the extension of generalized Mittag-Leffler function (\ref{eq4.7})  using above relation as:
\begin{equation}\label{eq4.22}
E_{\eta,\kappa}^{(\sigma;c)}(u;q ) = \sum\limits_{m=0}^{\infty} \frac{B_{q}(\sigma + m, c - \sigma)}{B_{q}(\sigma, c - \sigma)}~\frac{(q^{c};q)_{m}}{(q;q)_{m}}~
\frac{u^{m}}{\Gamma_{q}(\eta m + \kappa)}
\end{equation}
$$ (\Re(c)>\Re(\sigma)>0, |q|<1), $$
where $B_{q}(.)$ is the $q$-analog of beta function.
\vspace{.1cm}

We enumerate the relations and particular cases of $q$-analogue of extended generalized Mittag-Leffler function with other special functions as given below

\begin{enumerate}[label=(\roman*)]
\item If we put $c = 1$ in (\ref{eq4.22}), we obtain
\begin{equation}\label{eq4a}
E_{\eta,\kappa}^{(\sigma;1)}(u;q ) = \sum\limits_{m=0}^{\infty}  \frac{(q^{\sigma};q)_{m}}{(q;q)_{m}}~
\frac{u^{m}}{\Gamma_{q}(\eta m + \kappa)} = E_{\eta,\kappa}^{\sigma}(u;q ),
\end{equation}
where the function $E_{\eta,\kappa}^{\sigma}(u;q )$ is the $q$-analogue of Mittag-Leffler function defined in (\ref{eq4.7}) .
\item Again, if we take $\sigma = 1$ in (\ref{eq4.22}), we get
\begin{equation}\label{eq4b}
E_{\eta,\kappa}^{(1;c)}(u;q ) = \sum\limits_{m=0}^{\infty}
\frac{u^{m}}{\Gamma_{q}(\eta m + \kappa)} = e_{\eta,\kappa} (u;q ),
\end{equation}
the function $e_{\eta,\kappa} (u;q )$ can be termed as $q$-analogue of Mittag-Leffler function defined in (\ref{eq4.6}).
\item If we consider $\eta = \kappa = \sigma = 1 $, in (\ref{eq4.22}),
we find
\begin{equation}\label{eq4c}
E_{1,1}^{(1;c)}(u;q ) = \sum\limits_{m=0}^{\infty}  \frac{(q^{c};q)_{m}}{(q;q)_{m}}u^{m}
 =\frac{(q^{c}u;q)_{\infty}}{(q;q)_{\infty}} = {}_1\phi_0(q^{c};-;q,u),
\end{equation}
where the function ${}_2\phi_0(q^{c};-;q,u) = (1 - u)^{-c}$ can be termed as $q$- binomial function.
\item On setting $c = c+\sigma$, in (\ref{eq4.22}), then similarly, we obtain $q$-analogue of Mittag-Leffler function $E_{\eta,\kappa}^{\sigma}(u;q )$ defined in (\ref{eq4.7}).
\end{enumerate}

\section{Convergence of $E_{\eta,\kappa}^{(\sigma;c)}(u;q) $ }
\begin{thm}\label{1a}
 The $q$-analogue of the extended generalized Mittag-Leffler function defined by the summation formula (\ref{eq4.22}) converges absolutely for $|u|<(1 - q)^{-\eta}$ provided that $0 <q<1,\eta>0, \Re(c)>\Re(\sigma), c, \sigma \in \mathbb{C}$.
\end{thm}
\begin{proof}
Writing the summation formula (\ref{eq4.22}) as
$ E_{\eta,\kappa}^{(\sigma;c)}(u;q ) = \sum\limits_{m = 0}^{\infty}s_{m}$ and by applying ratio formula, we find
\begin{equation}\nonumber
\lim\limits_{m\rightarrow\infty}\bigg|\frac{s_{m+1}}{s_{m}}\bigg|
= \bigg|\frac{B_{q}(\sigma+m+1,c-\sigma)}{B_{q}(\sigma+m,c-\sigma)}\bigg|\bigg|\frac{(q^{c},q)_{m+1}}{(q^{c},q)_{m}}\bigg|\bigg|\frac{(q,q)_{m}}{(q,q)_{m+1}}\bigg|
\bigg|\frac{\Gamma(\eta m + \kappa)}{\Gamma(\eta m + \eta + \kappa)}~u\bigg|
\end{equation}
\begin{equation}\nonumber
= \lim\limits_{m\rightarrow\infty}\bigg|\frac{(q^{c+m},q)_{\infty}}{(q^{c+m+1},q)_{\infty}}~\frac{(q^{\sigma+m},q)_{\infty}}{(q^{\sigma+m+1},q)_{\infty}}~
\frac{(q^{\eta m+\kappa},q)_{\infty}}{(q^{\eta m+\kappa+\eta},q)_{\infty}}~\frac{(q^{m+1},q)_{\infty}}{(q^{m},q)_{\infty}}~~(1 - q)^{-\eta} u\bigg|
\end{equation}
\begin{equation}\nonumber
 = \lim\limits_{m\rightarrow\infty}\bigg|(1 - q^{c+m})~(1 - q^{\sigma+m})~(1 - q^{\eta m + \kappa})^{\eta}~
\frac{(1 - q)^{-\eta}}{(1 - q^{m})} ~u\bigg|
\end{equation}
\begin{equation}\label{eq4.23}
=|(1 - q)^{-\eta}| |u|, \quad for \quad 0<|q|<1.
\end{equation}
\end{proof}
\section{Recurrence Relations}
\begin{thm}\label{1b}
If $\eta, \kappa, \sigma \in \mathbb{C}$, $\Re(\eta)>0, \Re(\kappa)>0, \Re(\sigma)>0 $ and $\sigma \neq c $, then
\begin{equation}\nonumber
E_{\eta,\kappa}^{(\sigma;c)}(u;q ) = E_{\eta,\kappa}^{(\sigma + 1;c + 1)}(u;q ) - u~q^{c} ~E_{\eta,\eta + \kappa}^{(\sigma+ 1;c + 1)}(u;q ).
\end{equation}
\end{thm}
\begin{proof}
By the definition (\ref{eq4.22}), we have
\begin{equation}\nonumber
E_{\eta,\kappa}^{(\sigma;c)}(u;q ) = \sum\limits_{m=0}^{\infty} \frac{B_{q}(\sigma + m, c - \sigma)}{B_{q}(\sigma, c - \sigma)}~\frac{(q^{c};q)_{m}}{(q;q)_{m}}~
\frac{u^{m}}{\Gamma_{q}(\eta m + \kappa)},
\end{equation}
\begin{equation}\nonumber
\qquad  = \frac{1}{\Gamma(\kappa)}+ \sum\limits_{m=1}^{\infty} \frac{B_{q}(\sigma + m, c - \sigma)}{B_{q}(\sigma, c - \sigma)}~\frac{(1-q^{c})(q^{c+1};q)_{m-1}}{(q;q)_{m}}~
\frac{u^{m}}{\Gamma_{q}(\eta m + \kappa)}.
\end{equation}
\qquad Since $(1-q^{c}) = (1-q^{c+m})-q^{c}(1-q^{m})$, the above equation reduces to
\begin{equation}\nonumber
E_{\eta,\kappa}^{(\sigma;c)}(u;q ) = \frac{1}{\Gamma(\kappa)}+ \sum\limits_{m=1}^{\infty} \frac{B_{q}(\sigma + m, c - \sigma)}{B_{q}(\sigma, c - \sigma)}~\frac{(1-q^{c+m})(q^{c+1};q)_{m-1}}{(q;q)_{m}}~
\frac{u^{m}}{\Gamma_{q}(\eta m + \kappa)}-
\end{equation}
\begin{equation}\nonumber
- q^{c}~\sum\limits_{m=1}^{\infty} \frac{B_{q}(\sigma + m, c - \sigma)}{B_{q}(\sigma, c - \sigma)}~\frac{(1-q^{m})(q^{c+1};q)_{m-1}}{(q;q)_{m}}~
\frac{u^{m}}{\Gamma_{q}(\eta m + \kappa)}.
\end{equation}
On replacing $m$ with  $m+1$ in the second summation, it becomes
\begin{equation}\nonumber
E_{\eta,\kappa}^{(\sigma;c)}(u;q ) = \frac{1}{\Gamma(\kappa)}+ \sum\limits_{m=1}^{\infty} \frac{B_{q}(\sigma + m, c - \sigma)}{B_{q}(\sigma, c - \sigma)}~\frac{(q^{c+1};q)_{m}}{(q;q)_{m}}~
\frac{u^{m}}{\Gamma_{q}(\eta m + \kappa)}
\end{equation}
\begin{equation}\nonumber
\qquad- q^{c}~\sum\limits_{m=1}^{\infty} \frac{B_{q}(\sigma + m + 1, c - \sigma)}{B_{q}(\sigma, c - \sigma)}~\frac{ (q^{c+1};q)_{m}}{(q;q)_{m}}~
\frac{u^{m + 1}}{\Gamma_{q}(\eta m + \eta + \kappa)},
\end{equation}
which leads to the required result (\ref{1b}).
\end{proof}
\section{Some elementary properties of Extended $q$-Mittag-Leffler function}
We begin with the underlying theorem, which shows the integral representation of extended $q$-Mittag-Leffler function:
\begin{thm}\label{1c}(Integral representation)
For the extended $q$-Mittag-Leffler function, we have
\begin{equation}\label{eq4.24}
E_{\eta,\kappa}^{(\sigma;c)}(u;q ) =
\frac{1}{B_{q}(\sigma, c-\sigma)}\int_{0}^{1}t^{\sigma - 1}
\frac{(tq;q)_{\infty}}{(tq^{c-\sigma};q)_{\infty}}
E_{\eta,\kappa}^{(c)}(tu;q)~d_{q}t,
\end{equation}
provided that, $\eta, \kappa, \sigma \in \mathbb{C}$, $\Re(\eta)>0, \Re(\kappa)>0, \Re(\sigma)>0 $ and $\sigma \neq c $.
\end{thm}
\begin{proof}
By the definition of q-analogue of beta function, we can rewrite equation (\ref{eq4.22})    as follows:
\begin{equation}\nonumber
E_{\eta,\kappa}^{(\sigma;c)}(u;q ) = \sum\limits_{m=0}^{\infty}
\bigg\{ \int_{0}^{1}t^{\sigma+m-1}
\frac{(tq;q)_{\infty}}{(tq^{c-\sigma};q)_{\infty}}d_{q}t \bigg\}
\frac{1}{B_{q}(\sigma, c-\sigma)}
\end{equation}
\begin{equation}\nonumber
\times \frac{(q^{c};q)_{m}}{\Gamma_{q}(\eta m + \kappa)}
\frac{u^{m}}{(q;q)_{m}}
\end{equation}
\begin{equation}\nonumber
= \frac{1}{B_{q}(\sigma, c-\sigma)}\int_{0}^{1}t^{\sigma-1}
\frac{(tq;q)_{\infty}}{(tq^{c-\sigma};q)_{\infty}}d_{q}t~
\frac{(q^{c};q)_{m}}{(q;q)_{m}}
\frac{{tu}^{m}}{\Gamma_{q}(\eta m + \kappa)},
\end{equation}
which leads to the required result (\ref{eq4.24}).
\end{proof}
\begin{thm}
For $\eta, \kappa, \sigma \in \mathbb{C}, \Re(\eta)>0,
\Re(\kappa)>0, \Re(\sigma)>0, c\neq \sigma $, then for any $m \in \mathbb{N}$, we have
\begin{equation}\label{eq4.25}
D_{q}^{m}[u^{\kappa - 1} E_{\eta, \kappa}^{(\sigma;c)}(\lambda u^{\eta};q)] = u^{\kappa - m - 1}  E_{\eta, \kappa - m}^{(\sigma;c)}(\lambda u^{\eta};q).
\end{equation}
\end{thm}
\begin{proof}
By considering the function
\begin{equation}\nonumber
f(u) = u^{\kappa-1}  E_{\eta, \kappa}^{(\sigma;c)}(\lambda u^{\eta};q)
\end{equation}
 and using the definition (\ref{eq4.22}), then, in view of (\ref{eq4.19}), we obtain
\begin{equation}\nonumber
D_{q}^{m}[u^{\kappa - 1} E_{\eta, \kappa}^{(\sigma;c)}(\lambda u^{\eta})] = \sum\limits_{m=0}^{\infty} \frac{B_{q}(\sigma + m + 1, c - \sigma)}{B_{q}(\sigma, c - \sigma)}~\frac{ (q^{c };q)_{m}}{(q;q)_{m}}
\end{equation}
\begin{equation}\nonumber
\hspace{2.7cm}\times \frac{{\lambda^{m}}(1-q^{\eta m+\kappa-1})}{1-q} \frac{u^{\eta m + \kappa -2}}{\Gamma_{q}(\eta m + \kappa)}
\end{equation}

Since, according to the functional equation (\ref{eq4.18}), the r.h.s of the above expression can be written as
\begin{equation}\nonumber
 \sum\limits_{m=0}^{\infty} \frac{B_{q}(\sigma + m + 1, c - \sigma)}{B_{q}(\sigma, c - \sigma)}~\frac{ (q^{c };q)_{m}}{(q;q)_{m}} \frac{\lambda^{m} u^{\eta m + \kappa -2}}{\Gamma_{q}(\eta m + \kappa - 1)}=u^{\kappa-2}~E_{\eta, \kappa-1}^{(\sigma;c)}(\lambda u^{\eta};q).
\end{equation}

Conclusively, we obtain
\begin{equation}\nonumber
D_{q}^{m}[u^{\kappa - 1} E_{\eta, \kappa}^{(\sigma;c)}(\lambda u^{\eta};q)] = u^{\kappa - 2}  E_{\eta,\kappa-1}^{(\sigma;c)}(\lambda u^{\eta};q).
\end{equation}

Iterating above result $m-1$ times, we obtain the required result (\ref{eq4.25}).
\end{proof}
\begin{thm}
Let $\xi, \zeta, \sigma, \kappa \in \mathbb{C}; \Re(\xi), \Re(\kappa), \Re(\sigma)> 0; \zeta \neq 0, -1, -2, \ldots$ then
\begin{equation}\nonumber
\int_{0}^{1} u^{\xi - 1}(1 - qu)_{(\zeta - 1)} E_{\eta,\kappa}^{(\sigma;c)}(xu^{\rho};q ) d_{q}u  \hspace{5cm}
\end{equation}
\begin{equation}\label{eq4.26}
= \sum\limits_{m=0}^{\infty} \frac{B_{q}(\sigma + m, c - \sigma)(q^{c};q)_{m}}{B_{q}(\sigma, c - \sigma)(q;q)_{m}} \frac{x^{m} \Gamma_{q}(\xi + \rho m)\Gamma_{q}(\xi)}
{\Gamma_{q}(\eta m + \kappa)\Gamma_{q}(\xi + \zeta + \rho m )}.
\end{equation}
In particular,
\begin{equation}\label{eq4.27}
\int_{0}^{1} u^{\xi - 1}(1 - qu)_{(\zeta - 1)} E_{\eta,\kappa}^{(\sigma;c)}(xu^{\rho};q ) d_{q}u = \Gamma_{q}(\zeta) E_{\eta,\kappa+\zeta}^{(\sigma;c)}(x;q )  \hspace{1.4cm}
\end{equation}
\end{thm}

\begin{proof}
By using the definition (\ref{eq4.22}), the l.h.s of equation (\ref{eq4.26}) can be written as

\begin{equation}\nonumber
\int_{0}^{1} u^{\xi - 1}(1 - qu)_{(\zeta - 1)}\sum\limits_{m=0}^{\infty} \frac{B_{q}(\sigma + m, c - \sigma)(q^{c};q)_{m}}{B_{q}(\sigma, c - \sigma)(q;q)_{m}} \frac{u^{\rho m} x^{m}}{\Gamma_{q}(\eta m + \kappa)} d_{q}u\hspace{1cm}
\end{equation}

Interchanging the order of summation and integration and in view of equation (\ref{eq4.17}), we obtain the required result (\ref{eq4.26}).

In equation (\ref{eq4.26}) replacing $\rho = \eta $, $\kappa = \zeta$, then in view of equation (\ref{eq4.22}), we can clearly obtain (\ref{eq4.27}).
\end{proof}
\begin{thm}(q-Laplace transform)
For $q$-analogue of the extended generalized  Lapalce transform is defined as follows:
\begin{equation}\label{eq4.28}
_{q}L_{s}[E_{\eta,\kappa}^{(\sigma;c)}(xu^{\rho};q )] = \frac{1}{s}
\sum\limits_{m=0}^{\infty} \frac{B_{q}(\sigma + m, c - \sigma)(q^{c};q)_{m}}{B_{q}(\sigma, c - \sigma)(q;q)_{m}} \frac{ \Gamma_{q}(1 + \rho m)}{\Gamma_{q}(\eta m + \kappa)}
\end{equation}
\begin{equation}\nonumber
\times \bigg(\frac{(1-q)^{\rho} x}{s^{\rho}}\bigg)^{m}
\end{equation}
provided that $\kappa, \sigma, s \in \mathbb{C}; \Re(\beta), \Re(\kappa), \Re(s) > 0. $
\end{thm}
\begin{proof}
The $q$-Laplace transform of a suitable function is given by means of following $q$-integral \cite{Hw}
\begin{equation}\label{eq4.29}
_{q}L_{s}\{f(u)\} = \frac{1}{(1-q)}\int_{0}^{s^{-1}} E_{q}^{qsu}f(u)d_{q}u.
\end{equation}

The $q$-extension of the exponential function \cite{Gasp} is given by
\begin{equation}\label{eq4.30}
E_{q}^{u} =  _{0}\phi_{0}  (-,-; q, -u) = \sum\limits_{m = 0}^{\infty} \frac{q^{{{m}\choose{2}}} u^{m}}{(q;q)_{m}} = (-u;q)_{\infty}
\end{equation}
and
\begin{equation}\label{eq4.30I}
e_{q}^{u} = {_{0}\phi_{0}}(0,-; q, -u) = \sum\limits_{m = 0}^{\infty} \frac{u^{m}}{(q;q)_{m}} = \frac{1}{(u;q)_{\infty}},\quad |u|<1.
\end{equation}

By using the above $q$-exponential series  and the $q$-integral equation (\ref{eq4.19I}), we can write equation (\ref{eq4.29}) as
\begin{equation}\label{eq4.31}
_{q}L_{s}\{f(u)\} = \frac{(q;q)_{\infty}}{s} \sum_{j=0}^{\infty}
\frac{q^{j} f(s^{-1}q^{j})}{(q;q)_{j}}.
\end{equation}

Using the definition (\ref{eq4.22}) and the definition of $q$-Laplace transform, we obtain
\begin{equation}\nonumber
_{q}L_{s}[E_{\eta,\kappa}^{(\sigma;c)}(xu^{\rho};q )]= \frac{(q;q)_{\infty}}{s}\sum_{j=0}^{\infty}\frac{q^{j}}{(q;q)_{j}} \hspace{4cm}
\end{equation}
\begin{equation}\nonumber
\times \sum_{m=0}^{\infty}  \frac{B_{q}(\sigma + m, c - \sigma)}{B_{q}(\sigma, c - \sigma)} \frac{(q^{\sigma};q)_{m}}{(q;q)_{m}} \frac{[u(s^{-1} q^{j})^{\sigma}]^{m}}{ \Gamma_{q}(\eta m + \kappa)}.
\end{equation}

On interchanging the order of summation and writing the $j$ series as $_{1}\phi_{0}$, which can be summed up as
$\frac{1}{(q^{1 + \rho m};q)_{\infty}}$ and after some simplifications, we obtain the required result (\ref{eq4.28}).
\end{proof}
\section{Kober type fractional $q$- calculus operators  }
Agarwal \cite{Agr} established Kober type fractional $q$-integral operator in the following manner
\begin{equation}\label{eq7.1}
(I_{q}^{\nu,\mu} f)(u) = \frac{u^{-\nu-\mu}}{\Gamma_{q}(u)}
\int_{0}^{u}(u-tq)_{\mu-1} t^{\nu}f(t)d_{q}t,
\end{equation}
where $\Re(\mu)>0$.
\vspace{.1cm}
Also, Garg \emph{et al.} \cite{Gmc}  introduced Kober fractional $q$-derivative operator given by
\begin{equation}\label{eq7.2}
(D_{q}^{\nu,\mu} f)(u) = \prod_{i = 0}^{m}
\bigg([\nu + j]_{q} + uq^{\nu + j} D_{q}\bigg)
(I_{q}^{\nu+\mu, m - \mu} f)(u),
\end{equation}
where $m = [\Re(\mu)] + 1,\,m \in \mathbb{N}$.

The image formula of the power function $u^{m}$ under the above operators \cite{Gmc} are given as:

\begin{equation}\label{eq7.3}
I_{q}^{\nu, \mu} \{u^{m}\} =
\frac{\Gamma_{q}(\nu + m + 1)}{\Gamma_{q}(\nu +\mu + m + 1)}u^{m}
\end{equation}

\begin{equation}\label{eq7.4}
D_{q}^{\nu, \mu} \{u^{m}\} =
\frac{\Gamma_{q}(\nu +\mu + m + 1)}{\Gamma_{q}(\nu + m + 1)}u^{m}
\end{equation}
\begin{thm}\label{1d}
The underlying assumption holds true:
\begin{equation}\label{eq7.5}
I_{q}^{\nu,\mu}\{E_{\eta,\kappa}^{(\sigma;c)}(u;q)\}  = \sum\limits_{m=0}^{\infty} \frac{B_{q}(\sigma + m, c - \sigma)}{B_{q}(\sigma, c - \sigma)}~\frac{(q^{c};q)_{m}}{(q;q)_{m}}~
\frac{\Gamma_{q}(\nu + m + 1)}{\Gamma_{q}(\nu + \mu + m + 1)}
\frac{u^{m}}{\Gamma_{q}(\eta m + \kappa)},
\end{equation}
 particularly,
\begin{equation}\label{eq7.6}
I_{q}^{\nu,\mu} E_{\eta,\kappa}^{(\nu+\mu;1)}(u;q) =
\frac{\Gamma_{q}(\nu + 1)}{\Gamma_{q}(\nu + \mu + 1)}E_{\eta,\kappa}^{(\nu+ 1;1)}(u;q),
\end{equation}
provided that if $\eta, c >0, \kappa, \sigma, u \in \mathbb{C}; \Re(\kappa),\,\Re(\sigma)>0$.
\end{thm}
\begin{proof}
The proof of (\ref{eq7.5}) can easily be obtained by making use of the definition (\ref{eq4.22}) and the result (\ref{eq7.3}).

Now, on setting $\sigma = \nu + \mu $ in the definition (\ref{eq4.22}), we obtain the result (\ref{eq7.6}).
\end{proof}

\begin{thm}\label{1e}
The underlying assumption holds true:
\begin{equation}\label{eq7.7}
D_{q}^{\nu,\mu} \{E_{\eta,\kappa}^{(\sigma;c)}(u;q )\} = \sum\limits_{m=0}^{\infty} \frac{B_{q}(\sigma + m, c - \sigma)}{B_{q}(\sigma, c - \sigma)}~\frac{(q^{c};q)_{m}}{(q;q)_{m}}~
\frac{\Gamma_{q}(\nu + \mu + m + 1)}{\Gamma_{q}(\nu + m + 1)}
\frac{u^{m}}{\Gamma_{q}(\eta m + \kappa)},
\end{equation}
 particularly,
\begin{equation}\label{eq7.8}
D_{q}^{\nu,\mu} E_{\eta,\kappa}^{(\nu+1;1)}(u;q) =
\frac{\Gamma_{q}(\nu + \mu + 1)}{\Gamma_{q}(\nu + 1)}E_{\eta,\kappa}^{\nu+ \mu}(u;q)
\end{equation}
provided that if $\eta, c >0, \kappa, \sigma, u \in \mathbb{C}; \Re(\kappa),\Re(\sigma)>0$.
\end{thm}
\begin{proof}
The proof of (\ref{eq7.7}) can easily be obtained by making use of the definition (\ref{eq4.22}) and the result (\ref{eq7.4}).~Similarly, on setting $\sigma = \nu + 1 $ in the definition (\ref{eq4.22}), we obtain the result (\ref{eq7.8}).
\end{proof}

\end{document}